\documentclass[a4paper,11pt]{article}
\usepackage[T1]{fontenc}
\usepackage[utf8]{inputenc}
\usepackage[english]{babel}
\usepackage{bm,amsmath,amssymb}
\usepackage{mathtools}
\usepackage{amsthm}
\usepackage{courier}
\usepackage{a4wide}
\usepackage{graphicx}
\usepackage{tikz}
\usepackage{subfigure}
\usepackage{units}
\usepackage{xspace}
\usepackage{multirow}
\usepackage{appendix}
\usepackage{mathrsfs}
\usepackage{xcolor}
\usepackage[bookmarks=true]{hyperref}

\newcommand{\ag}{\alpha}

\newcommand{\dg}{\delta}
\newcommand{\eps}{\varepsilon}

\newcommand{\sg}{\sigma}

\newcommand{\FF}{\mathcal{F}}
\newcommand{\GG}{\mathcal{G}}
\newcommand{\UU}{\mathcal{U}}
\DeclareMathOperator{\sgn}{sgn}

\theoremstyle{plain}
\newtheorem{theorem}{Theorem}
\newtheorem{lemma}[theorem]{Lemma}
\newtheorem{proposition}[theorem]{Proposition}

\author{%
  Pierre~Boutaud\footnote{Université Paris-Saclay, CNRS,  Laboratoire de mathématiques d'Orsay, 91405, Orsay, France.
    E-mail: \texttt{pierre.boutaud at u-psud dot fr}}
  \and 
  Pascal~Maillard\footnote{Institut de Mathématiques de Toulouse, CNRS, UMR5219, Université de Toulouse, 118 route de Narbonne, F-31062 Toulouse cedex 09, France. E-mail: \texttt{Pascal.Maillard at math dot univ-toulouse dot fr}}}

\title{Seneta-Heyde norming for branching random walks with $\alpha$-stable spine}

\date{April 6, 2020}

\begin{document}

\maketitle

\begin{abstract}
We consider branching random walks with a spine in the domain of attraction of an $\alpha$-stable Lévy process. For this process, the classical derivative martingale in general degenerates in the limit. We first determine the quantity replacing the derivative martingale and show that it converges to a non-degenerate limit under a certain $L\log L$-type condition which we assume to be optimal. We go on to give the Seneta-Heyde norming for the critical additive martingale under the same assumptions. The proofs are based on the methods introduced in our previous paper which considered the finite variance case [Boutaud and Maillard (2019), \textit{EJP}, vol. 24, paper no. 99].

\medskip


\medskip

\end{abstract}

\section{Introduction}\label{sec:intro}
\subsection*{Definitions and results}

We consider discrete-time real-valued branching random walks (BRWs), which can be informally described as follows. At time $n=0$, we start with one initial particle at the origin. Then, at each time step $n\ge 1$, every particle dies and gives birth to a random, possibly infinite number of particles distributed randomly on the real line. More precisely, the children of a particle at position $x\in\mathbb{R}$ are positioned at $x+X_1,x+X_2,\ldots$, where the vector $(X_1,X_2,\ldots)$ follows a given law $\Theta$, called the offspring distribution of the branching random walk. At each generation, the reproduction events are independent. Also, it is possible for several particles to share the same position. We further assume that the Galton-Watson process formed by the number of particles at each generation is super-critical, so that the system survives with positive probability.

The last decade has seen considerable interest in the \emph{extremal particles} in branching random walks, as well as in related models such as the two-dimensional Gaussian Free Field, Gaussian multiplicative chaos and characteristic polynomials of certain random matrices, see e.g. \cite{ShiLectureNotes,ZeitouniLNBRW,RhodesVargasReview,BovierBook} for fairly recent reviews. A fundamental tool in the study of the extremes of the branching random walk is the so-called spine decomposition which allows to represent the branching random walk after a change of measure as another branching process involving a special particle called the \emph{spine} evolving as a certain random walk. Almost all results in the literature have been obtained under the assumption that this ``spinal'' random walk has finite variance. It is natural to wonder what happens when this condition is not satisfied, which is the goal of the present article.

Formally, the branching random walk can be constructed as a stochastic process indexed by the Ulam-Harris tree $\UU= \bigcup_{n\ge0} (\mathbb{N}^*)^n$, where $\mathbb{N}^* = \{1,2,\ldots\}$. Particles are identified with vertices $u\in\UU$, i.e.~words over the alphabet $\mathbb{N}^*$. The length of the word $u$, i.e.~the generation of the particle, is denoted by $|u|$. The position of the particle $u$ is denoted by $X_u$. If the particle indexed by $u$ does not exist, we set $X_u = +\infty$. The branching random walk described above then defines a process $(X_u)_{u\in\UU}$ taking values in $\bar{\mathbb{R}} = \mathbb{R}\cup\{+\infty\}$ and the offspring distribution $\Theta$ is a probability distribution on $(\bar{\mathbb{R}})^{\mathbb{N}^*}$. We further convene that mathematical expressions such as sums or products over the set $\{|u| = n\}$ of particles at generation $n$ are meant to ignore those $u$ for which $X_u = +\infty$. Furthermore, we use the convention $\sum_{\emptyset}=0$ and $\prod_{\emptyset}=1$.

As mentioned above, we assume that the branching is super-critical, i.e.
\begin{equation}\label{eq:supercritical}
\mathbb{E}\left[\sum_{|u|=1} 1\right] > 1.
\end{equation}
We recall the definition of the so-called boundary case \cite{BK2005}:
\begin{equation}\label{eq:boundarycase}
\mathbb{E}\left[\sum_{|u|=1}e^{-X_{u}}\right]=1\quad\text{and}\quad\mathbb{E}\left[\sum_{|u|=1}X_{u}e^{-X_{u}}\right]=0,
\end{equation}
where it is implicitly assumed that the second expectation is well-defined.

It is well known that under \eqref{eq:boundarycase}
\begin{equation*}
W_{n}=\sum_{|u|=n}e^{-X_{u}},\qquad D_{n}=\sum_{|u|=n}X_{u}e^{-X_{u}},
\end{equation*}
are martingales with respect to the canonical filtration of the branching random walk $\FF_n=\sg(X_u,|u|\le n)$. We will refer to $(W_n)_{n\ge0}$ as the \emph{additive martingale} or \emph{Biggins' martingale}, in reference to Biggins \cite{Biggins1977} and to $(D_n)_{n\ge0}$ as the \emph{derivative martingale}. Since $(W_n)$ is a non-negative martingale it converges almost surely and the second equality in \eqref{eq:boundarycase} implies that the limit is $0$ (see Biggins \cite{Biggins1977} and Lyons \cite{Lyons1998}). The discussion over the rate at which $W_n$ converges to $0$ is referred to in the literature as the \emph{Seneta-Heyde norming} for the branching random walk. 

Under the additional assumption
\begin{equation}\label{eq:variance}\sigma^2:=\mathbb{E}\left[\sum_{|u|=1}X_u^2e^{-X_u}\right]\in(0,\infty),\end{equation}
Biggins and Kyprianou  \cite{Biggins2004} showed that $D_{n}$ converges a.s. to a finite nonnegative limit $D_\infty$. We will refer to this setting as the finite variance case.

Aïdékon and Shi \cite{Aidekon2014} proved under \eqref{eq:variance} that $\sqrt{n}W_n$ converges almost surely to $\sqrt{\frac{2}{\pi\sigma^2}}D_\infty$. Moreover, Chen \cite{Chen2015} showed that the limit is non-degenerate if and only if the following assumptions hold (sufficiency was shown before by Aïdékon \cite{Aidekon2013}):
\begin{align}
\mathbb{E}\left[W_{1}\log_{+}^{2}W_{1}\right]&<\infty\label{eq:Wlog2W},\\
\mathbb{E}\left[\tilde{W}_{1}\log_{+}\tilde{W}_{1}\right]&<\infty,
\end{align}
where $\tilde{W}_{1}=\sum_{|u|=1}X_{u}^{+}e^{-X_u}$, $x^+=x\vee 0$ and $\log_+(x)=\log(x\vee 1)$.

He, Liu and Zhang \cite{He2016} then considered a different setting, using the following assumptions:
\begin{align}
\mathbb{E}\left[\sum_{|u|=1}e^{-X_u}\boldsymbol 1_{X_u\le -x}\right]&=o(x^{-\ag})\quad(x\to\infty)\label{eq:lefttailHLZ}\\
\mathbb{E}\left[\sum_{|u|=1}e^{-X_u}\boldsymbol 1_{X_u\ge x}\right]&\sim \frac{c}{x^\ag}\quad(x\to\infty)\label{eq:righttailHLZ}\\
\mathbb{E}\left[W_1\left(\log_+W_1\right)^\ag+\tilde{W}_1\left(\log_{+}\tilde{W}_1\right)^{\ag-1}\right]&<\infty,\label{eq:LlogLHLZ}
\end{align}
where $\ag\in(1,2)$ and $c>0$. Under these assumptions they proved that $D_n$ still converges to $D_\infty$ and $n^{1/\ag}W_n$ converges to $cD_\infty$ with an explicit constant $c>0$, thus extending Aïdékon and Shi's result \cite{Aidekon2014}.

\medskip

Our goal in this article is to significantly generalize the above results by using and developing the toolbox introduced in Boutaud and Maillard \cite{Boutaud2019} for the finite variance case.

\paragraph{Assumptions.}
In what follows, we assume \eqref{eq:supercritical} and $\mathbb{E}\left[\sum_{|u|=1}e^{-X_u}\right]=1$. Let $(S_n)_{n\in\mathbb{N}}$ denote a real-valued random walk with $S_0 = 0$ and the law of its increments given by 
\begin{equation*}
\forall A\text{ measurable set},\quad \mathbb{P}(S_1\in A)=\mathbb{E}\left[\sum_{|u|=1}e^{-X_u}\boldsymbol 1_{X_u\in A}\right].
\end{equation*}
We suppose that there exists $\ag\in(0,2)\backslash\{1\}$ and a sequence $(a_n)_n$, such that $\frac{S_{n}}{a_n}$ converges to an $\ag$-stable distribution as $n\to\infty$, with characteristic function
\begin{equation}\label{eq:caracLevy}
t\mapsto \exp\left(-\lambda|t|^\ag\exp\left(-i\frac{\pi\theta\alpha}{2}\sgn(t)\right)\right),\ \lambda>0,\ |\theta|\le 1\wedge\left(\frac{2}{\alpha}-1\right),\ |\theta|\ne1.
\end{equation}

One consequence of the assumption $\mathbb{E}\left[\sum_{|u|=1}e^{-X_u}\right]=1$ is that the minimum $\min_{|u|=n}X_u$ tends to $\infty$ as $n\to\infty$ almost surely on the event of survival of the branching random walk \cite[Theorem 3]{Biggins1998}. We will use this fact without further mention during the remainder of the article.

The assumptions on the parameters have been chosen such as to match those of Heyde \cite{Heyde1969}, Bingham \cite{Bingham1973a} and Emery \cite{Emery1972}. Note that, under these assumptions, the random walk $(S_n)_n$ oscillates. One could extend our results to more general parameters, for example to $\ag=1$ using recent results by Berger \cite{Berger2018}, but we will not do so here for the sake of simplicity.

The parametrization in \eqref{eq:caracLevy} corresponds to the form (C) from Zolotarev \cite{Zolotarev} and has been chosen here such as to simplify the constants appearing in our results. See below for its relation to other parametrizations of stable laws.

Define the negativity parameter 
\begin{equation}\label{eq:defbarrho}
\bar\rho=\frac{1-\theta}{2}\in(0,1),
\end{equation}
and the positivity parameter $\rho=1-\bar\rho$.

The first theorem introduces the right candidate to replace the derivative martingale $D_n$ defined above. Recall that we use the notation $x^+=x_+=x\vee 0$.

\begin{theorem}\label{th:convergeZn}
Under the assumptions stated above, define
\begin{equation}
Z_n=\sum_{|u|=n}\left(X_u^+\right)^{\alpha\bar\rho}e^{-X_u}.
\end{equation}
Then $Z_n$ converges almost surely to a non-negative limit $Z_\infty$.  If, moreover, the following assumption holds:
\begin{equation}\label{eq:momentsWandZ}
\mathbb{E}\left[W_1\left(\log_{+}W_1\right)^{\alpha}+Z_1\left(\log_{+}Z_1\right)^{\alpha\rho}\right]<\infty,
\end{equation}
then $Z_\infty$ is (strictly) positive almost surely on the event of survival of the branching random walk.
\end{theorem}
We believe the assumption \eqref{eq:momentsWandZ} to be optimal for this result in the sense that $Z_\infty=0$ almost surely otherwise, similarly to the finite variance case \cite{Chen2015}.

We can now state our main result, the Seneta-Heyde norming of the additive martingale:
\begin{theorem}\label{th:mainresult}
Under the same assumptions as in Theorem~\ref{th:convergeZn}, we have:
\begin{equation}
a_{n}^{\alpha\bar\rho}W_n\underset{n\to\infty}{\longrightarrow}\frac{\kappa}{\lambda^{\bar\rho}} Z_\infty\quad\text{in probability},
\end{equation}
where $\kappa$ is a positive constant depending on $\alpha$ and $\bar\rho$ only (see equation~\eqref{eq:kappameander} below for an expression).
\end{theorem}

\paragraph{Comments on the derivative martingale.}
Note that if $\ag\bar\rho=1$, since $\min_{|u|=n}X_u\to\infty$ as $n\to\infty$, $Z_n$ and $D_n$ have the same limit. So in that case, $D_n$ is the right quantity to study the convergence of $W_n$. However, when $\ag\bar\rho<1$, $D_n$ is no longer the right quantity and will in fact tend to $\infty$ almost surely.

\paragraph{Comments on the expression of $\kappa$.}
Let $\mathcal{X}=(\mathcal{X}_t)_{t\ge0}$ be the $\alpha$-stable Lévy process starting at 0 and such that $\mathcal{X}_1$ has characteristic function given by \eqref{eq:caracLevy}, with $\lambda=1$. One can give an expression of $\kappa$ in terms of $\mathcal{X}$. Let $\mathbb{P}^{(m)}$ be the law of the meander of length $1$, associated with $\mathcal{X}$. That is
\begin{equation}
\mathbb{P}^{(m)}\left((\mathcal{X}_t)_{t\in[0,1]}\in A\right)=\lim_{x\to0}\mathbb{P}\left((\mathcal{X}_t+x)_{t\in[0,1]}\in A\mid \inf_{t\in[0,1]}X_t+x\ge0\right).
\end{equation}
Then, we have
\begin{equation}\label{eq:kappameander}
\kappa=\frac{1}{\mathbb{E}^{(m)}\left[\mathcal{X}_1^{\alpha\bar\rho}\right]}.
\end{equation}
We show in Section \ref{sec:computingkappa} that $\kappa$ has a simple explicit expression when $\alpha\in (1,2]$ and the Lévy process $(\mathcal{X}_t)_{t\ge0}$ has no positive jumps, i.e.~when $\alpha\rho=1$:
\begin{equation}\label{eq:kappaexplicit}
\kappa=\frac{1}{\Gamma(\alpha)\Gamma(1/\ag)}.
\end{equation}

\paragraph{Another parametrization of stable laws.}
In the literature one usually parametrizes the stable laws defined above such that the characteristic function is of the form
\begin{equation}\tag{A}\label{eq:formA}
t\mapsto\exp\left(-\lambda'|t|^\ag(1-i\sgn(t)\beta\tan(\pi\alpha/2)\right),
\end{equation}
with $\ag\in(0,2)\backslash\{1\}$, $\beta\in[-1,1]$, $\lambda'>0$, and $|\beta|<1$ if $\ag<1$. This corresponds to form (A) in Zolotarev \cite{Zolotarev}. The parameters are related in the following way:
\begin{align}
\beta&=\cot\left(\frac{\pi\ag}{2}\right)\tan\left(\frac{\pi\theta\ag}{2}\right)\label{eq:betatheta}\\
\lambda'&=\lambda\cos\left(\frac{\pi\theta\alpha}{2}\right).\label{eq:lambda'lambda}
\end{align}
This gives an expression of the negativity parameter in terms of the parameters of form \eqref{eq:formA}:
\begin{equation}
\bar\rho=\frac{1}{2}-\frac{1}{\pi\ag}\arctan\left(\beta\tan\left(\frac{\pi\ag}{2}\right)\right).
\end{equation}

\subsection*{Illustration of the assumptions}

Many aspects of the asymptotic behavior of branching random walks are encoded in the following function \cite{Biggins1977a}:
$$\varphi(t)=\log\mathbb{E}\left[\sum_{|u|=1}e^{-tX_u}\right],\quad t\in\mathbb R$$
Note that the function $t\mapsto \varphi(t+1)$ is the log-Laplace transform of $S_1$. 
In this section, we illustrate the assumptions from this article in terms of this function.

Recall that, under the assumptions stated above, we only consider the cases where the random walk $(S_n)_n$ oscillates and we exclude the case $\alpha=1$ which would require a specific treatment. To be more precise, we consider the following cases:
\begin{enumerate}
\item[(a)] the finite variance case $\alpha=2$,
\item[(b)] $\alpha\in(1,2)$ and $\mathcal{X}$ has no positive jumps,
\item[(b')] $\alpha\in(1,2)$ and $\mathcal{X}$ has no negative jumps,
\item[(c)] $\alpha\in(0,2)\backslash\{1\}$ and $\mathcal{X}$ has positive and negative jumps.
\end{enumerate}

\begin{figure}[ht]
	\centering
	\begin{minipage}{0.45\textwidth}
	\centering
		\begin{tikzpicture}
			\begin{small}
				\def\r{0.5}
				\draw[thin, ->]	(-2,0) -- (2,0);
				\draw[thin, ->]	(0,0) -- (0,4*\r+.5);
				\foreach \x in {0}
					\draw (\x,.05)--(\x,-.05) node[below] {$1$};
				\draw [very thick, domain=-1.5:1.5] plot (\x, {\r*\x*\x});
			\end{small}
		\end{tikzpicture}
	\caption{Case (a)}
	\end{minipage}\hfill
	\begin{minipage}{0.45\textwidth}
	\centering
		\begin{tikzpicture}
			\begin{small}
				\def\r{0.5}
				\draw[thin, ->]	(-2,0) -- (2,0);
				\draw[thin, ->]	(0,0) -- (0,4*\r+.5);
				\foreach \x in {0}
					\draw (\x,.05)--(\x,-.05) node[below] {$1$};
				\draw (.05,4*\r)--(-.05,4*\r) node[above left] {$\infty$} [dotted]	(-2,4*\r) -- (2,4*\r);
				\draw [very thick, domain=-1.5:0] plot (\x, {\r*\x*\x});
				\draw [very thick, domain=0:2] plot (\x, 4*\r);
			\end{small}
		\end{tikzpicture}
	\caption{Case (b)}
	\end{minipage}\hfill
\end{figure}

\begin{figure}[ht]
\centering
	\begin{minipage}{0.45\textwidth}
		\centering
		\begin{tikzpicture}
			\begin{small}
				\def\r{0.5}
				\draw[thin, ->]	(-2,0) -- (2,0);
				\draw[thin, ->]	(0,0) -- (0,4*\r+.5);
				\foreach \x in {0}
					\draw (\x,.05)--(\x,-.05) node[below] {$1$};
				\draw (.05,4*\r)--(-.05,4*\r) node[above left] {$\infty$} [dotted]	(-2,4*\r) -- (2,4*\r);
				\draw [very thick, domain=-2:0] plot (\x, 4*\r);
				\draw [very thick, domain=0:1.5] plot (\x, {\r*\x*\x});
			\end{small}
		\end{tikzpicture}
	\caption{Case (b')}
	\end{minipage}\hfill
	\begin{minipage}{0.45\textwidth}
		\centering
		\begin{tikzpicture}
			\begin{small}
			\def\r{0.5}
			\draw[thin, ->]	(-2,0) -- (2,0);
			\draw[thin, ->]	(0,0) -- (0,4*\r+.5);
			\foreach \x in {0}
				\draw (\x,.05)--(\x,-.05) node[below] {$1$};
			\draw (.05,4*\r)--(-.05,4*\r) node[above left] {$\infty$} [dotted]	(-2,4*\r) -- (2,4*\r);
			\draw [very thick, domain=-2:0] plot (\x, 4*\r);
			\node at (0,0)[circle,fill,inner sep=1.5pt]{};
			\draw [very thick, domain=0:2] plot (\x, 4*\r);
			\draw[fill=white] (0,4*\r) circle (2pt);
			\end{small}
		\end{tikzpicture}
		\caption{Case (c)}
	\end{minipage}\hfill
\end{figure}

\begin{figure}[ht]
	\centering
	\begin{minipage}{0.45\textwidth}
		\centering
		\begin{tikzpicture}
			\begin{small}
				\def\r{0.5}
				\draw[thin, ->]	(-2,0) -- (2,0);
				\draw[thin, ->]	(0,0) -- (0,4*\r+.5);
				\foreach \x in {0}
					\draw (\x,.05)--(\x,-.05) node[below right] {$1$};
				\draw (.05,4*\r)--(-.05,4*\r) node[above left] {$\infty$} [dotted]	(-2,4*\r) -- (2,4*\r);
				\draw [very thick, domain=-2:0] plot (\x, {\r*4*(-sqrt(-\x)+0.5*\x*\x});
				\draw [very thick, domain=0:2] plot (\x, 4*\r);
			\end{small}
		\end{tikzpicture}
	\caption{$\alpha < 1$, no negative jumps (not considered here)}
	\end{minipage}\hfill
	\begin{minipage}{0.45\textwidth}
		\centering
		\begin{tikzpicture}
			\begin{small}
				\def\r{0.5}
				\draw[thin, ->]	(-2,0) -- (2,0);
				\draw[thin, ->]	(0,0) -- (0,4*\r+.5);
				\foreach \x in {0}
					\draw (\x,.05)--(\x,-.05) node[below] {$1$};
				\draw (.05,4*\r)--(-.05,4*\r) node[above left] {$\infty$} [dotted]	(-2,4*\r) -- (2,4*\r);
				\draw [very thick, domain=-1.5:0] plot (\x, {\r*\x*(\x-1)});
				\draw [very thick, domain=0:2] plot (\x, 4*\r);
			\end{small}
		\end{tikzpicture}
	\caption{$\alpha>1$ and $\mathbb{E}[S_1]\ne 0$ (not considered here)}
	\end{minipage}
\end{figure}

These cases are illustrated in Figures 1 to 4. These contain schematic plots of the function $\varphi$ which give rise to each of the four cases (a),(b),(b'),(c) stated above. Note that even in cases (a),(b),(b'), the function $\varphi$ might be infinite at any point $t\ne 1$ --- however, in all cases it may not be finite at points at which the function is infinite as plotted. Indeed, if $\varphi(t)<\infty$ for some $t<1$ ($t>1$), then the law of $S_1$ must have exponentially decaying right (left) tails, which implies that the stable process has no positive (negative) jumps.

While cases (b) and (b') are natural cases to consider, the case (c) might seem degenerate. Indeed, the growth of the branching random walk at exponential scale is entirely determined by the function $\varphi$ through its Fenchel-Legendre transform \cite{Biggins1977a}, which is here simply a linear function. However, such branching random walks might arise naturally as limits of certain sequences of branching random walks, though we are not aware of any specific example. In all cases (b),(b') and (c), the finer asymptotic behavior of such BRWs has not been considered before to our knowledge and is therefore open for investigation.

We exclude in this article the cases where $(S_n)$ drifts towards $+\infty$ or $-\infty$, which is left open for investigation. This can happen either when $\alpha < 1$ and $\mathcal{X}$ has only positive or only negative jumps or when $\alpha \in (1,2]$ and $\mathbb{E}[S_1]\ne 0$. The behavior of the function $\varphi$ in these two cases ist schematically depicted in Figures 5 and 6, respectively. The asymptotic of the minimum in the latter case has been studied in \cite{Barral2014}.

\paragraph{Overview of the article.}
Throughout the remainder of the article, we suppose $\lambda=1$ which can be obtained by replacing $a_n$ by $a_n\lambda^{1/\ag}$.
The article is organized as follows. Section~\ref{sec:spine} contains preliminaries about the spinal decomposition and the associated many-to-one formula. Section~\ref{sec:renewalandminimum} contains the definition of a renewal function, description of its behaviour and how it impacts the tail of the minimum of the random walk $(S_n)_n$ as well as a truncated first moment. Section~\ref{sec:proofZnconverge} contains the proof of Theorem~\ref{th:convergeZn}. Section~\ref{sec:proofmain} contains the proof of our main result (Theorem~\ref{th:mainresult}) and the proof of two lemmas are deferred to Section~\ref{sec:proofT1T2}. Finally, Section~\ref{sec:computingkappa} contains the calculations to obtain the explicit expression of $\kappa$ in \eqref{eq:kappaexplicit}.

\paragraph{Acknowledgements}
We thank Vladimir Vatutin for discussions on random walks and for bringing to our attention the article \cite{Vatutin2017}.

\section{The spinal decomposition}\label{sec:spine}

In this section, we recall a change of measure and an associated spinal decomposition of the BRW due to Lyons~\cite{Lyons1998}. It will be helpful to allow the initial particle of the BRW to sit at an arbitrary position $x\in\mathbb{R}$, this will be denoted by adding the subscript $x$ as in $\mathbb{P}_x$ and $\mathbb{E}_x$ (if $x=0$, the subscript is ignored). Then $(W_n)_{n\ge0}$ is still a non-negative martingale with $W_0 = e^{-x}$. Denote by $\mathcal{F}_n=\sigma(X_u,|u|\le n)$ the canonical filtration of the BRW and define $\FF_{\infty}=\bigvee_{n\ge0}\FF_{n}$. Using Kolmogorov's extension theorem, for every $x\in\mathbb{R}$, there exists a probability measure $\mathbb{P}^*_{x}$ on $\FF_{\infty}$ such that for every generation $n\ge 0$,
\begin{equation}
\frac{d\mathbb{P}^{*}_{x}}{d\mathbb{P}_x}\Big|_{\FF_{n}}=e^{x}W_{n}.
\end{equation}

Following Lyons \cite{Lyons1998} we see $\mathbb{P}_x^*$ as the projection to $\FF_{\infty}$ of a probability (also denoted $\mathbb{P}_x^*$) defined on a bigger probability space equipped with a so-called \emph{spine}, a distinguished ray in the tree. We will denote the vertex on the spine at generation $n$ by $\xi_{n}$ and its position by $X_{\xi_n}$. The spinal BRW evolves as follows under $\mathbb{P}^*_{x}$:
\begin{itemize}
\item Start at generation 0 with one particle $\xi_0$ at position $x$.
\item At generation $n$, all particles except $\xi_{n}$ reproduce according to the point process $\Theta$ and $\xi_{n}$ reproduces according to the size-biased reproduction law $\Theta^*$ defined by
\[
\frac{d\Theta^*}{d\Theta}(x_1,x_2,\ldots) = \sum_{i\ge1} e^{-x_i}.
\]
\item The spine at generation $n+1$ is chosen amongst the children $u$ of $\xi_{n}$ with probability proportional to $e^{-X_{u}}$.
\end{itemize}
The following \emph{many-to-one formula} can be deduced from Lyons \cite{Lyons1998}, see e.g.~Aïdékon~\cite{Aidekon2013}.
\begin{proposition}[Many-to-one formula]\label{prop:MTOformula}
For any $x\ge 0$, $n\in\mathbb{N}=\{0,1,\dots\}$ and every family $\left(H_{n}(u)\right)_{u\in\UU}$ of $\FF_{n}$-measurable non-negative random variables, one has
\begin{equation}
\mathbb{E}_{x}\left[\sum_{|u|=n}e^{-X_{u}}H_{n}(u)\right]=e^{-x}\mathbb{E}_{x}^{*}\left[H_{n}(\xi_{n})\right].
\end{equation}
\end{proposition}

The spinal decomposition implies that the process $(X_{\xi_{n}})_{n\in\mathbb{N}}$ follows the law of a random walk under $\mathbb{P}^*_{x}$ (whose increments do not depend on $x$), which we will refer to as the spinal random walk. It has the same law as the random walk $(S_n)_n$ from the introduction (with $S_0 = x$). In particular, by assumption, $(X_{\xi_n})_n$ is in the domain of attraction of an $\ag$-stable process.

\section{Renewal function and the tail of the minimum}\label{sec:renewalandminimum}

Let $R$ be the renewal function associated with the strictly descending ladder heights of the random walk $(S_n)$. Then $R$ is harmonic for the sub-Markov process obtained by killing $(S_n)$ when entering $(-\infty,0)$ \cite{Tan1989}, i.e.~
\begin{equation}\label{eq:Rharmonic}
\forall x\ge0, R(x)=\mathbb{E}_x\left[R(S_1)\boldsymbol 1_{S_1\ge0}\right].
\end{equation}

We proceed by giving some bounds and asymptotical results on the renewal function $R$. These results come from the works of Heyde \cite{Heyde1969}, Bingham \cite{Bingham1973a} and Emery \cite{Emery1972}. We remark that they use the parametrization \eqref{eq:formA} of stable laws, with moreover $\beta$ replaced by $-\beta$. Using the formulae \eqref{eq:betatheta} and \eqref{eq:lambda'lambda}, one can easily translate their results in terms of the parameters $\theta$ and $\lambda$. Recall that in this section and the following sections, we assume that $\lambda=1$.

From the definition of the parameters, we can easily see that $\ag\bar\rho\le 1$. Heyde showed in \cite[equation (15)]{Heyde1969} that
\begin{equation}\label{eq:defc1}
s^{-\ag\bar\rho}\left(1-\mathbb{E}[\exp(-sH_1)]\right)\underset{s\to0}{\longrightarrow}1.
\end{equation}
where $H_1$ is the first strictly descending ladder height of the walk. To read this off Heyde's formula, note that $\lambda'\sqrt{1+\beta^2\tan^2(\pi\ag/2)}=\lambda=1$, by the trigonometric identity $1+\tan^2=1/\cos^2$.

We distinguish between the cases $\ag\bar\rho=1$ and $\ag\bar\rho<1$ (recall that $\ag\bar\rho=1$ occurs if and only if the Lévy process $\mathcal{X}$ has no negative jumps).
Following Bingham \cite[end of section 5]{Bingham1973}, if $\ag\bar\rho=1$, then $\mathbb E[H_1]=1$ and the key renewal theorem gives us
\begin{equation*}
R(x)\sim x\quad(x\to\infty).
\end{equation*}
On the other hand, if $\ag\bar\rho<1$, then $H_1$ has infinite mean and lies in the domain of attraction of a stable subordinator and
\begin{equation*}
\mathbb{P}(H_1\ge x)\sim \frac{1}{x^{\ag\bar\rho}\Gamma(1-\ag\bar\rho)},
\end{equation*}
which gives, by Feller (\cite{Feller1971} XIV (3.3)):
\begin{equation*}
R(x)\sim \frac{x^{\ag\bar\rho}}{\Gamma(1+\ag\bar\rho)}.
\end{equation*}

Since putting $\ag\bar\rho=1$ in this last expression provides the right asymptotics for that case, we can define $c_{1}=\frac{1}{\Gamma(1+\ag\bar\rho)}$, such that in both cases, as $x\to\infty$,
\begin{equation}\label{eq:eqR}
R(x)\sim c_{1}x^{\ag\bar\rho},
\end{equation}
and as a consequence, there exists $c_{1}'>0$ such that for all $x\ge 0$,
\begin{equation}\label{eq:boundR}
R(x)\le c_{1}'(1+x)^{\ag\bar\rho}\le c_{1}'(1+x)
\end{equation}

From \eqref{eq:boundR}, we easily obtain the following lemma:
\begin{lemma}
\label{lem:sousadditif}
For all $x,y\in\mathbb{R}$, we have
\begin{equation*}
R(x+y)\le c_{1}'(1+x^+)^{\ag\bar\rho}(1+y^+)^{\ag\bar\rho}\le c_{1}'(1+x^+)(1+y^+).
\end{equation*}
\end{lemma}

The corresponding results hold for $\hat{R}$, the renewal function associated with the striclty ascending ladder heights, by replacing $\bar\rho$ by $\rho$.

We now turn to the tail of the minimum $\min_{k\le n} S_k$. We have the following theorem:

\begin{theorem}\label{th:ScalingMin}
Under the assumptions on $(S_n)$ stated in the introduction, and recalling the definition of $\bar\rho$ in \eqref{eq:defbarrho}, 
we have for every $x\ge0$ that
\begin{equation}\label{eq:minasymptotics}
\mathbb{P}_x\left(\min_{k\le n}S_k\ge 0\right)\sim\frac{K}{a_n^{\alpha\bar\rho}}R(x)\quad(n\to\infty),
\end{equation}
where $K=\kappa/c_{1}$, with $\kappa$ defined in \eqref{eq:kappameander}. In particular, for all $n\ge 0$,
\begin{equation}\label{eq:boundmin}
\mathbb{P}_x\left(\min_{k\le n}S_k\ge 0\right)\le\frac{c_2K}{a_n^{\alpha\bar\rho}}R(x)\quad(n\to\infty),
\end{equation}
where $c_2>0$ is a constant.
\end{theorem}

\begin{proof}[Proof of Theorem~\ref{th:ScalingMin}]
~\\
Replacing $S_n$ by $-S_n$ in Theorem 2 of Bingham \cite{Bingham1973a} (which uses results from Emery \cite{Emery1972}), gives us that there exists some slowly varying function $L(n)$ such that for every continuity point $x$ of $R$, we have
\begin{equation}\label{eq:BinghamMin}
\mathbb{P}_x\left(\min_{k\le n}S_k\ge 0\right)\sim\frac{R(x)}{n^{\bar\rho}L(n)\Gamma(\rho)}\quad(n\to\infty).
\end{equation}
This also holds for every $x\ge0$ by applying the same method as in Kozlov \cite{Kozlov1976}.

Moreover, we have the following asymptotic, which can be deduced from equations (3.5), (3.6) and (3.11) in Caravenna and Chaumont \cite{Caravenna2008}, see also Vatutin and Dyakonova \cite[equations (13),(14)]{Vatutin2017}:
\begin{equation*}
\mathbb{P}_{0}\left(\min_{k\le n}S_k\ge 0\right)\sim\frac{\kappa}{R(a_n)}\quad(n\to\infty),
\end{equation*}
where $\kappa$ has the expression given in equation \eqref{eq:kappameander}.
Putting $x=0$ in equation \eqref{eq:BinghamMin} and comparing with this last equivalent yields
\begin{align*}
\frac{1}{n^{\bar\rho}L(n)\Gamma(\rho)}&=\frac{R(0)}{n^{\bar\rho}L(n)\Gamma(\rho)}\\
&\sim\frac{\kappa}{R(a_n)}\quad(n\to\infty)\\
&\sim\frac{\kappa}{c_{1}a_n^{\alpha\bar\rho}}\quad(n\to\infty)\\
&=\frac{K}{a_n^{\alpha\bar\rho}},
\end{align*}
which, together with \eqref{eq:BinghamMin}, gives \eqref{eq:minasymptotics} and ends the proof.
\end{proof}

\section{Proof of Theorem~\ref{th:convergeZn}}\label{sec:proofZnconverge}

In this section, we prove Theorem~\ref{th:convergeZn}. We first show that $Z_n$ converges to a non-negative limit $Z_\infty$. The key to this will be the following martingale. Define
\begin{equation}\label{eq:quasiDn}
D_n'=\sum_{|u|=n}R(X_u)e^{-X_u}\boldsymbol 1_{\min_{v\le u}X_u\ge0}
\end{equation}
Using the fact that $R$ is harmonic for the spinal random walk killed below zero \eqref{eq:Rharmonic} and the many-to-one lemma (Proposition~\ref{prop:MTOformula}), one easily shows that $(D_n')_n$ is a martingale with respect to the canonical filtration of the branching random walk. 

Following Kyprianou \cite{Kyprianou2004}, we introduce a ``barrier'' by defining for $a\ge0$ the quantities
\begin{align}
Z_n^{(a)}&=\sum_{|u|=n}\left(X_u^+\right)^{\ag\bar\rho}e^{-X_u}\boldsymbol 1_{\forall v\le u, X_v\ge-a}\\
D_n'^{(a)}&=\sum_{|u|=n}R(X_u+a)e^{-X_u}\boldsymbol 1_{\forall v\le u, X_v\ge-a}.
\end{align}
Note that for fixed $a$, $D_n'^{(a)}$ under $\mathbb{P}$ has the same law as $e^{a}D_n'$ under $\mathbb{P}_a$. Since $(D_n')_n$ is a non-negative martingale under $\mathbb{P}_a$, it follows that $D_n'^{(a)}$ converges almost surely to some non-negative random variable.

Moreover, since $\min_{|u|=n}X_u\to\infty$ almost surely as $n\to\infty$ (see introduction), using \eqref{eq:eqR} gives us that for all $a>0$, $c_{1}Z_n^{(a)}$ and $D_n'^{(a)}$ have the same limit almost surely.
Finally we have for all $a>0$
\begin{align*}
\mathbb{P}(Z_n\text{ converges})&\ge \mathbb{P}(Z_n\text{ converges and }\forall k,Z_k=Z_k^{(a)})\\
&=\mathbb{P}(Z_n^{(a)}\text{ converges and }\forall k,Z_k=Z_k^{(a)})\\
&=\mathbb{P}(\forall k, Z_k=Z_k^{(a)}),
\end{align*}
which tends to $1$ as $a\to\infty$, again by the fact that $\min_{|u|=n}X_u\to\infty$ almost surely as $n\to\infty$. So $Z_n$ converges almost surely to a non-negative random variable $Z_\infty$.

We now show that the limit $Z_\infty$ is non-trivial under the additional assumption \eqref{eq:momentsWandZ}. To do so we use the martingale $D_n'$ defined in \eqref{eq:quasiDn}.

By \eqref{eq:boundR}, we have for all $n$, $D_n'\le Z_n$ a.s., so $\mathbb{P}(D'_\infty>0)\le\mathbb{P}(Z_\infty>0)$ and it suffices to prove that $\mathbb{P}(D'_\infty>0)>0$.

First note that $\mathbb{E}[D'_0]=R(0)=1>0$. We will prove that $D'_n$ is uniformly integrable which will give that $\mathbb{E}[D_\infty']=1$, $\mathbb{P}(D'_\infty>0)>0$ and so $Z_\infty$ will not be trivial.
Furthermore, by standard arguments for branching processes, one shows that $\mathbb{P}(Z_\infty=0)$ is a fixed point of the generating function of the number of offspring of a branching random walk particle and therefore is equal to the extinction probability of the branching random walk.

Following Chen \cite{Chen2015}, we will state a specific case of Theorem 2.1(i) by Biggins and Kyprianou \cite{Biggins2004} that provides a sufficient condition for the non-triviality of $D'_\infty$. First we need to make a change of measure to condition the walk $(X_{\xi_n})$ to stay non-negative at all times:
recalling the harmonicity of $R$ for the spinal random walk killed below $0$ under $\mathbb{P}^*$, namely equation~\eqref{eq:Rharmonic}, we can define for all $x\ge 0$ a probability measure $\mathbb{P}^+_x$ by
\begin{equation}
\frac{d\mathbb{P}_x^+}{d\mathbb{P}_x^*}\Big|_{\mathcal \FF_n} = \frac{1}{R(x)} R(X_{\xi_n}) \boldsymbol 1_{\min_{k\le n}X_{\xi_{k}}\ge0}.
\end{equation}
Under $\mathbb{P}_x^+$, the random walk $(X_{\xi_n})_n$ can be seen as being conditioned to stay non-negative at all times, see Bertoin and Doney \cite{Bertoin1994}. We will denote by $\mathbb{E}^+_x$ the associated expectation.

We can now state the theorem:

\begin{theorem}[See Theorem~2.1(i) in Biggins and Kyprianou \cite{Biggins2004}]\label{th:BK04}
Define a random variable $Q$ such that for all $x\ge0$, under $\mathbb{P}_x$,
\begin{equation*}
Q=\frac{\sum_{|u|=1}R(X_u)e^{-X_u}\boldsymbol 1_{X_u\ge0}}{R(x)e^{-x}}.
\end{equation*}
Suppose
\begin{equation}\label{eq:condition}
\sum_{n=1}^{\infty}\mathbb{E}_{X_{\xi_n}}\left[Q\left((R(X_{\xi_n})e^{-X_{\xi_n}}Q)\wedge 1\right)\right]<\infty\quad\mathbb{P}^+-\text{a.s.}
\end{equation}
Then
\begin{equation}
\mathbb{E}[D_\infty']=R(0)=1.
\end{equation}
\end{theorem}


We will now check that condition \eqref{eq:condition} holds in order to conclude that $D_\infty'$ is non-trivial.

By \eqref{eq:eqR} and \eqref{eq:boundR}, there exists $c_{3}>0$, such that
\begin{align}
Q\le\frac{\sum_{|u|=1}R(X_u)e^{-X_u}}{R(x)e^{-x}}&= \frac{\sum_{|u|=1}R(X_u-x+x)e^{-(X_u-x+x)}}{R(x)e^{-x}}\nonumber\\
&\le c_{3}\frac{\sum_{|u|=1}\left(R(x)+(X_u-x)^{\ag\bar\rho}_+\right)e^{-(X_u-x)}}{R(x)}\nonumber\\
&= c_{3}\left(W_1+\frac{Z_1}{R(x)}\right),\label{eq:biasedbound}
\end{align}
where we recall that
\begin{equation*}
Z_1=\sum_{|u|=1}(X_u^+)^{\ag\bar\rho}e^{-X_u}.
\end{equation*}
Using \eqref{eq:biasedbound}, we have for all $n\ge1$,
\begin{align}
&\mathbb{E}_{X_{\xi_n}}\left[Q\left((R(X_{\xi_n})e^{-X_{\xi_n}}Q)\wedge1\right)\right]\nonumber\\
&\le (c_{3})^2\mathbb{E}_{X_{\xi_n}}\left[\left(W_1+\frac{Z_1}{R(X_{\xi_n})}\right)\left((W_1R(X_{\xi_n})e^{-X_{\xi_n}}+Z_1e^{-X_{\xi_n}})\wedge1\right)\right]\nonumber\\
&\le (c_{3})^2 g(X_{\xi_n}),\label{eq:boundedbyg}
\end{align}
where $g(y)=\mathbb{E}\left[\left(W_1+\frac{Z_1}{R(y)}\right)\left((W_1R(y)e^{-y}+Z_1e^{-y})\wedge1\right)\right]$. Using again \eqref{eq:boundR}, there exists $c_{4}>0$, such that $$\forall y\ge0, R(y)e^{-y}\le c_{1}'(1+y)^{\alpha\bar\rho}e^{-y}\le c_{4} e^{-y/2}.$$ We then obtain the following bound on $g$: define $f_1$ and $f_2$ as
\begin{align}
\forall y\ge 0, f_1(y) &= \mathbb{E}\left[W_1\left( e^{-y/4}(W_1+Z_1)\wedge 1\right)\right]\label{eq:deff1}\\
\text{and}\quad f_2(y) &= \mathbb{E}\left[Z_1\left( e^{-y/4}(W_1+Z_1)\wedge 1\right)\right]\label{eq:deff2}
\end{align}
(the exponent $y/4$ is chosen for later convenience). Then, for some $c_{5}>0$,
\begin{equation}\label{eq:gf1f2}
\forall y\ge 0, g(y)\le c_{5}\left(f_1(2y)+\frac{f_2(2y)}{R(y)}\right).
\end{equation}
We now want to use a Tauberian theorem in the spirit of Boutaud and Maillard \cite{Boutaud2019}. To do so we first state the following lemma.

\begin{lemma}\label{lem:HLZ41}
If
\begin{equation*}
\mathbb{E}\left[W_1\left(\log_+W_1\right)^{\ag}+Z_1\left(\log_+Z_1\right)^{\ag\rho}\right]<\infty,
\end{equation*}
then
\begin{equation*}
\mathbb{E}\left[W_1\left(\log_+(W_1+Z_1)\right)^{\ag}+Z_1\left(\log_+(W_1+Z_1)\right)^{\ag\rho}\right]<\infty.
\end{equation*}
\end{lemma}
The proof of this lemma is delayed at the end of this section.

Using Lemma~\ref{lem:HLZ41}, we obtain with assumption \eqref{eq:momentsWandZ} that
\begin{equation*}
\mathbb{E}\left[W_1\left(\log_{+}(W_1+Z_1)\right)^{\ag}\right]<\infty\quad\text{and}\quad\mathbb{E}\left[Z_{1}\left(\log_{+}(W_1+Z_1)\right)^{\ag\rho}\right]<\infty.
\end{equation*}
We then apply Lemma C.1 from Boutaud and Maillard \cite{Boutaud2019} twice to the r.v. $W_1+Z_1$, once under the law $\mathbb{E}[W_1\cdot]$ and with the regularly varying function $x\mapsto x^{\ag-1}$, and once under the law $\mathbb{E}[(Z_1/\mathbb{E}[Z_1])\cdot]$ and with the regularly varing function $x\mapsto x^{\ag\rho-1}$ whose index is strictly greater than $-1$. We then obtain
\begin{equation}\label{eq:f1f2integralbound}
\int_{0}^\infty f_{1}(y)y^{\ag-1} dy<\infty\quad\text{and}\quad\int_{0}^{\infty}f_{2}(y)y^{\ag\rho-1}dy<\infty.
\end{equation}
Equation \eqref{eq:f1f2integralbound} and the bound \eqref{eq:boundR} on $R$ give,
\begin{equation}
\int_{0}^{\infty}\left(f_{1}(2y)+\frac{f_2(2y)}{R(y)}\right)y^{\ag-1}dy<\infty.
\end{equation}
Such an integral condition is all we need in order to conclude thanks to the following lemma:

\begin{lemma}\label{lem:GreenEstimate}
Let $f:\mathbb{R}^+\mapsto\mathbb{R}^+$ be a bounded, non-increasing function satisfying $$\int_{0}^{\infty}y^{\ag-1}f(y)dy<\infty.$$
Then,
\begin{equation*}
\mathbb{E}_{x}^+\left[\sum_{n=0}^{\infty}f(X_{\xi_n})\right]\to 0,\quad\text{as }x\to\infty.
\end{equation*}
Furthermore, the above expectation is finite for every $x\ge 0$.
\end{lemma}
The proof of this lemma is postponed to the end of this section.

Noticing that $y\in\mathbb{R}^+\mapsto f_{1}(2y)+\frac{f_2(2y)}{R(y)}\in\mathbb{R}^+$ is bounded and non-increasing, we obtain by Lemma~\ref{lem:GreenEstimate} and the bound \eqref{eq:f1f2integralbound} that
\begin{equation}
\forall x\ge0,\quad \mathbb{E}^+_x\left[\sum_{n=0}^{\infty}g(X_{\xi_n})\right]<\infty.
\end{equation}

So finally, going back to equation \eqref{eq:boundedbyg}, summing over $n$ and taking expectations, we obtain:
\begin{align*}
\mathbb{E}^+\left[\sum_{n=1}^{\infty}\mathbb{E}_{X_{\xi_n}}\left[Q\left((R(X_{\xi_n})e^{-X_{\xi_n}}Q)\wedge1\right)\right]\right] \le (c_{3})^2\mathbb{E}^+\left[\sum_{n=1}^{\infty}g(X_{\xi_n})\right]<\infty.
\end{align*}
This proves that condition \eqref{eq:condition} holds so, using Theorem~\ref{th:BK04}, $D_\infty'$ is non-trivial.

\begin{proof}[Proof of Lemma~\ref{lem:HLZ41}]	
~\\
Let $x,y\ge0$. We want to bound $x\left(\log_+y\right)^\ag$. We can assume that $y>1$, otherwise the bound is trivial.

If $y<x^2$, we have $x\left(\log_+y\right)^\ag\le2^\ag x\left(\log_+x\right)^{\ag}$. On the other hand, if $y\ge x^2$ (and $y>1$), we have
\begin{align*}
x\left(\log_+y\right)^\ag&\le y^{1/2}\left(\log_+y\right)^\ag\\
&=y^{1/2}\left(\log_+y\right)^{\ag\bar\rho}\left(\log_+y\right)^{\ag\rho}\\
&=y^{1/2}\left(2\ag\bar\rho\log_+\left(y^{\frac{1}{2\ag\bar\rho}}\right)\right)^{\ag\bar\rho}\left(\log_+y\right)^{\ag\rho}\\
&\le \left(2\ag\bar\rho\right)^{\ag\bar\rho}y\left(\log_+y\right)^{\ag\rho},
\end{align*}
using the inequality $\log_+z \le z$.
Using these two bounds, we obtain that
\begin{equation*}
\mathbb{E}\left[W_1\left(\log_+Z_1\right)^{\ag}\right]\le2^\ag\mathbb{E}\left[W_1\left(\log_+W_1\right)^{\ag}\right]+\left(2\ag\bar\rho\right)^{\ag\bar\rho}\mathbb{E}\left[Z_1\left(\log_+Z_1\right)^{\ag\rho}\right]<\infty.
\end{equation*}

Then, using the inequality
\begin{equation}\label{eq:boundylogx}
\forall x,y\ge0,\ y\left(\log_+x\right)^{\ag\rho}\le\max\left\{x\left(\log_+x\right)^{\ag\rho},y\left(\log_+y\right)^{\ag\rho}\right\},\end{equation}
we obtain that
\begin{equation*}
\mathbb{E}\left[Z_1\left(\log_+W_1\right)^{\ag\rho}\right]<\infty.
\end{equation*}

The rest of the proof is in the exact same spirit as the proof of Lemma B.1 (ii) in Aïdékon \cite{Aidekon2013}.
\end{proof}

\begin{proof}[Proof of Lemma~\ref{lem:GreenEstimate}]
~\\
We adapt the proof of Lemma 5.2 in Boutaud and Maillard \cite{Boutaud2019}, using a different decomposition of the integral. From the definition of $\mathbb{P}_x^+$, we have
\begin{equation}\label{eq:Green+to*}
\mathbb{E}_{x}^+\left[\sum_{n=0}^{\infty}f(X_{\xi_n})\right]=\frac{1}{R(x)}\mathbb{E}_{x}^*\left[\sum_{n=0}^{\infty}R(X_{\xi_n})f(X_{\xi_n})\boldsymbol 1_{\forall k\le n, X_{\xi_k}\ge 0}\right].
\end{equation}

Let $\mu$ and $\hat{\mu}$ be the renewal measures associated to (the absolute values of) the strictly descending and strictly ascending ladder heights of $(X_{\xi_n})_{n\ge 0}$, respectively. Recall that $R(x)=\mu([0,x])$ and $\hat{R}(x)=\hat{\mu}([0,x])$. Using Theorem~A.3 from Boutaud and Maillard \cite{Boutaud2019}, we obtain
\begin{align}
\mathbb{E}_{x}^*\left[\sum_{n=0}^{\infty}R(X_{\xi_n})f(X_{\xi_n})\boldsymbol 1_{\forall k\le n, X_{\xi_k}\ge 0}\right]&=c^=\int_{z=0}^{\infty}\int_{y=0}^{x}R(x-y+z)f(x-y+z)\mu(dy)\hat{\mu}(dz)\nonumber\\
&=:I(x),\label{eq:defIx}
\end{align}
where $c^==\exp\left(\sum_{n=1}^{\infty}\frac{1}{n}\mathbb{P}(X_{\xi_n}=0)\right)>0$.

Define $\tilde{f}:y\mapsto(1+y)^{\ag\bar\rho}f(y)$, $y\ge 0$. Using inequality \eqref{eq:boundR}, we know there exists a constant $c_{6}=c^=c_{1}'>0$ such that
\begin{equation}\label{eq:boundI}
I(x)\le c_{6}\int_{z=0}^{\infty}\int_{y=0}^{x}\tilde{f}(x-y+z)\mu(dy)\hat{\mu}(dz).
\end{equation}

We will integrate first over $z$ and then over $y$ in order to bound $I(x)$. For $w\ge 0$, define $g(w)=\int_{0}^{\infty}\tilde{f}(w+z)\hat{\mu}(dz)$. Then \eqref{eq:boundI} implies
\begin{equation}\label{eq:boundIinty}
I(x)\le c_{6}\int_{0}^{x}g(x-y)\mu(dy).
\end{equation}
We decompose $g$ using a geometric pattern. We have for every $w\ge 0$, using that $f$ is non-increasing,
\begin{align*}
g(w)&=\int_{0}^{1}\tilde{f}(w+z)\hat{\mu}(dz)+\sum_{k=0}^{\infty}\int_{(2^k,2^{k+1}]}\tilde{f}(w+z)\hat{\mu}(dz)\\
&\le\int_{0}^{1}\tilde{f}(w+z)\hat{\mu}(dz)+\sum_{k=0}^{\infty}\left(1+2^{k+1}\right)^{\ag\bar\rho}f\left(w+2^k\right)\left(\hat{R}(2^{k+1})-\hat{R}(2^{k})\right),
\end{align*}
since $f$ is non-increasing. Then using that $\hat{R}$ is positive and the inequality $\hat{R}(x)\le \hat{c_{1}}'(1+x)^{\ag\rho}$, which is obtained in the same way as for $R$ by replacing $\bar\rho$ by $\rho$, we have
\begin{align*}
g(w)&\le\int_{0}^{1}\tilde{f}(w+z)\hat{\mu}(dz)+\sum_{k=0}^{\infty}\left(1+2^{k+1}\right)^{\ag\bar\rho}f\left(w+2^k\right)\hat{R}(2^{k+1})\\
&\le \int_{0}^{1}\tilde{f}(w+z)\hat{\mu}(dz)+\hat{c_{1}}'\sum_{k=0}^{\infty}2^{(k+2)\ag\bar\rho}f\left(w+2^k\right)2^{(k+2)\ag\rho}\\
&=\int_{0}^{1}\tilde{f}(w+z)\hat{\mu}(dz)+\hat{c_{1}}' 2^{3\ag}\sum_{k=0}^{\infty}2^{(k-1)\ag}f(w+2^k)\quad(\text{since }\rho+\bar\rho=1)\\
&\le \int_{0}^{1}\tilde{f}(w+z)\hat{\mu}(dz)+\hat{c_{1}}' 2^{3\ag}\sum_{k=0}^{\infty}\frac{1}{2^{k-1}}\int_{2^{k-1}}^{2^k}z^\ag f(w+z)dz,
\end{align*}
since $f$ is non-increasing. Finally we rewrite this bound as follows
\begin{align*}
g(w)&\le \int_{0}^{1}\tilde{f}(w+z)\hat{\mu}(dz)+\hat{c_{1}}' 2^{3\ag}\sum_{k=0}^{\infty}\frac{2^k}{2^{k-1}}\int_{2^{k-1}}^{2^{k}}z^{\ag-1} f(w+z)dz\\
&\le \int_{0}^{1}\tilde{f}(w+z)\hat{\mu}(dz)+\hat{c_{1}}' 2^{3\ag+1}\sum_{k=0}^{\infty}\int_{2^{k-1}}^{2^{k}}z^{\ag-1} f(w+z)dz\\
&= \int_{0}^{1}\tilde{f}(w+z)\hat{\mu}(dz)+\hat{c_{1}}' 2^{3\ag+1}\int_{1/2}^{\infty}z^{\ag-1}f(w+z)dz.
\end{align*}
The integral condition on $f$ and the fact that $f$ is non-increasing ensure that $g$ is bounded and, by dominated convergence, 
\begin{equation}\label{eq:gto0}
g(w)\to 0,\quad w\to\infty.
\end{equation}

By \eqref{eq:Green+to*}, \eqref{eq:defIx} and \eqref{eq:boundIinty}, it remains to show that
\begin{equation}\label{eq:convergenceGreenEstimate}
\frac{1}{R(x)}\int_{0}^{x}g(x-y)\mu(dy)\to 0,\quad x\to\infty.
\end{equation}
Let $\dg>0$. By \eqref{eq:gto0}, there exists $y_0\ge 0$ such that $\forall y\ge y_0, g(y)\le\dg$. Then,
\begin{equation*}
\int_{0}^{x}g(x-y)\mu(dy)\le\dg\mu([0,x])+\int_{x-y_0}^{x}g(x-y)\mu(dy).
\end{equation*}
The first term on the right-hand side is $\dg R(x)$ by definition and the second term is uniformly bounded by a constant, as a consequence of the key renewal theorem (Feller \cite{Feller1971} p363, note that it applies even if $R(x)$ grows sublinearly), and the boundedness of $g$. As a consequence,
\begin{equation*}
\limsup_{x\to\infty}\frac{1}{R(x)}\int_{0}^{x}g(x-y)\mu(dy)\le \dg.
\end{equation*}
Since $\dg$ was arbitrary, this proves \eqref{eq:convergenceGreenEstimate} and thus finishes the proof.
\end{proof}

\section{Proof of Theorem~\ref{th:mainresult}}\label{sec:proofmain}

We define the following quantities that will appear in the proof of Theorem~\ref{th:mainresult}, for $n, k_0\ge 0$ :
\begin{align}
W_{n}'&=\sum_{|u|=n}e^{-X_{u}}\boldsymbol{1}_{\min_{v\le u}X_v\ge0}\label{eq:WnPrime}\\
W_{n,k_0}''&=\sum_{|u|=n}e^{-X_{u}}\boldsymbol{1}_{\min_{v\le u, |v|\ge k_0}X_v\ge0}
\end{align}

We have $\min_{|u|=n}X_u\to\infty$ a.s., thus
\begin{equation}\label{eq:compareWnsecond}
\forall \varepsilon>0,\exists k_0:\mathbb{P}\left(\forall n, W_{n,k_0}''=W_n\right)>1-\varepsilon.
\end{equation}

From Theorem~\ref{th:ScalingMin}, we can deduce the asymptotics for $\mathbb{E}_{x}[W_n']$ :
\begin{proposition}\label{prop:momentWnprime}
Let $x\ge 0$. Then, as $n\to\infty$,
\begin{equation}
\mathbb{E}_x[W_n']\sim\frac{K}{a_n^{\alpha\bar\rho}}R(x)e^{-x},
\end{equation}
and for all $n\ge 0$,
\begin{equation}
\mathbb{E}_{x}[W_n']\le \frac{c_{2}K}{a_n^{\alpha\bar\rho}}R(x)e^{-x},
\end{equation}
where $K$ and $c_{2}$ are the constants from Theorem~\ref{th:ScalingMin} and $R$ is the renewal function associated with the strictly descending ladder heights of the random walk $(X_{\xi_n})_n$.
\end{proposition}

\begin{proof}[Proof of Proposition~\ref{prop:momentWnprime}]
\begin{align*}
\mathbb{E}_{x}[W_n']&=\mathbb{E}_{x}\left[\sum_{|u|=n}e^{-X_u}\boldsymbol 1_{\min_{v\le u}X_v\ge0}\right]\\
&=e^{-x}\mathbb{P}_x^{*}\left(\min_{k\le n}X_{\xi_k}\ge 0\right),\quad\text{by the many-to-one formula.}
\end{align*}
Applying Theorem~\ref{th:ScalingMin} to the random walk $(X_{\xi_n})_n$ ends the proof.
\end{proof}

\begin{proof}[Proof of Theorem~\ref{th:mainresult}]
~\\
Remember that we assume $\lambda=1$ which can be obtained by replacing $a_n$ by $a_n\lambda^{1/\alpha}$.

We will now apply the method from Boutaud and Maillard \cite{Boutaud2019} to our setting and make slight changes when needed.

We will start by proving that for any $s>0$, $\mathbb{E}[\exp(-s a_n^{\alpha\bar\rho}W_{n,k_0}'')|\FF_{k_0}]$ converges in probability to $\exp(-s \kappa Z_\infty)$ as first $n$, then $k_0$, tend to infinity. We do so by proving a lower and upper bound. We then use Cantor diagonal extraction and apply a lemma from \cite{Boutaud2019} translating the convergence of conditional Laplace transform in term of convergence in probability. We then use Equation~\eqref{eq:compareWnsecond} to conclude.

We start by the lower bound on the conditional Laplace transform. For any $s>0$, we have
\begin{align}
\mathbb{E}\left[\exp\left(-s a_n^{\alpha\bar\rho} W_{n,k_0}''\right)\Big{|}\FF_{k_0}\right]&=\prod_{|u|=k_0} \mathbb{E}_{X_{u}}\left[\exp\left(-s a_n^{\alpha\bar\rho} W_{n-k_0}'\right)\right]\nonumber\\
&\ge \exp\left(-s a_n^{\alpha\bar\rho} \sum_{|u|=k_0}\mathbb{E}_{X_u}\left[W_{n-k_0}'\right]\right),\label{eq:jensen}
\end{align}
by Jensen's inequality.

By Proposition~\ref{prop:momentWnprime}, for every $x\in\mathbb{R}$, $a_n^{\alpha\bar\rho} \mathbb{E}_{x}[W_{n-k_0}']$ converges to $KR(x)e^{-x}$ as $n\to\infty$ and is bounded from above by $c_{2}KR(x)e^{-x}$. Furthermore, using the many-to-one formula, one easily checks that $\sum_{|u|=k_0}R(X_u)e^{-X_u}$ is finite in expectation and therefore almost surely. By dominated convergence, we get almost surely
\begin{equation}\label{eq:converge}
a_n^{\alpha\bar\rho}\sum_{|u|=k_0}\mathbb{E}_{X_u}\left[W_{n-k_0}'\right]\underset{n\to\infty}{\longrightarrow}K\sum_{|u|=k_0}R(X_u)e^{-X_u}.
\end{equation}
By Equations \eqref{eq:jensen} and \eqref{eq:converge}, we get almost surely,
\begin{equation}\label{eq:lower}
\underset{n\to\infty}{\liminf}\ \mathbb{E}\left[\exp\left(-s a_n^{\alpha\bar\rho} W_{n,k_0}''\right)\Big{|}\FF_{k_0}\right]\ge \exp\left(-s K\sum_{|u|=k_0}R(X_u)e^{-X_u}\right).
\end{equation}
Since $\min_{|u|=k_0}X_u\to\infty$ almost surely as $k_0\to\infty$, using \eqref{eq:eqR} and Theorem~\ref{th:convergeZn}, we get, almost surely,
\begin{equation}
\label{eq:lower2}
\lim_{k_0\to\infty}\sum_{|u|=k_0}KR(X_u)e^{-X_u}=\lim_{k_0\to\infty}\sum_{|u|=k_0}Kc_{1}\left(X_u^+\right)^{\ag\bar\rho}e^{-X_u}=\kappa Z_\infty,
\end{equation}
so that, by \eqref{eq:lower} and \eqref{eq:lower2}, almost surely,
\begin{equation}
\label{eq:lower3}
\liminf_{k_0\to\infty} \underset{n\to\infty}{\liminf}\ \mathbb{E}\left[\exp\left(-s a_n^{\alpha\bar\rho} W_{n,k_0}''\right)\Big{|}\FF_{k_0}\right] \ge e^{-s\kappa Z_\infty}.
\end{equation}

We now deal with the upper bound. For $s>0$ fixed,  and any $s'\in(0,s)$ there exists $\varepsilon>0$ such that 
\begin{equation}\label{eq:deltabound}
\forall x\in[0,\varepsilon),e^{-s x}\le 1-s' x.
\end{equation}

Fix $s>0$ and $s'\in(0,s)$ (that will tend to $s$), and take $\varepsilon$ satisfying \eqref{eq:deltabound}. We compute
\begin{align*}
\mathbb{E}\left[\exp\left(-s a_n^{\alpha\bar\rho}W_{n,k_0}''\right)\Big{|}\FF_{k_0}\right]&=\prod_{|u|=k_0}\mathbb{E}_{X_u}\left[\exp\left(-s a_n^{\alpha\bar\rho}W_{n-k_0}'\right)\right]\\
&\le\prod_{|u|=k_0}\mathbb{E}_{X_u}\left[\exp\left(-s a_n^{\alpha\bar\rho}W_{n-k_0}'\boldsymbol 1_{a_n^{\alpha\bar\rho}W_{n-k_0}'<\varepsilon}\right)\right],
\end{align*}
since $W_{n-k_0}'$ is non-negative. Using inequality~\eqref{eq:deltabound}, the linearity of expectation and finally the inequality $1-x\le e^{-x}$, we compute:
\begin{align*}
\mathbb{E}\left[\exp\left(-s a_n^{\alpha\bar\rho}W_{n,k_0}\right)\Big{|}\FF_{k_0}\right]&\le\prod_{|u|=k_0}\mathbb{E}_{X_u}\left[1-s' a_n^{\alpha\bar\rho}W_{n-k_0}'\boldsymbol 1_{a_n^{\alpha\bar\rho}W_{n-k_0}'<\varepsilon}\right]\\
&\le\prod_{|u|=k_0}\left(1-s'\mathbb{E}_{X_u}\left[a_n^{\alpha\bar\rho}W_{n-k_0}'\boldsymbol 1_{a_n^{\alpha\bar\rho}W_{n-k_0}'<\varepsilon}\right]\right)\\
&\le\exp\left(-\sum_{|u|=k_0}s'\mathbb{E}_{X_u}\left[a_n^{\alpha\bar\rho}W_{n-k_0}'\boldsymbol 1_{a_n^{\alpha\bar\rho}W_{n-k_0}'<\varepsilon}\right]\right).
\end{align*}

Using Fatou's lemma, we obtain
\begin{align}
&\limsup_{n\to\infty}\mathbb{E}\left[\exp(-s a_n^{\alpha\bar\rho}W_{n,k_0}'')\Big{|}\FF_{k_0}\right]\nonumber\\
&\le\exp\left(-\sum_{|u|=k_0}s'\liminf_{n\to\infty}\mathbb{E}_{X_u}\left[a_n^{\alpha\bar\rho}W_{n-k_0}'\boldsymbol 1_{a_n^{\alpha\bar\rho}W_{n-k_0}'<\varepsilon}\right]\right)\nonumber\\
&=\exp\left(s'\sum_{|u|=k_0}\left(-\liminf_{n\to\infty}\mathbb{E}_{X_{u}}\left[a_n^{\alpha\bar\rho} W_{n-k_0}'\right]
+\limsup_{n\to\infty}\mathbb{E}_{X_{u}}\left[a_n^{\alpha\bar\rho} W_{n-k_0}'\boldsymbol{1}_{a_n^{\alpha\bar\rho} W_{n-k_0}'\ge \varepsilon}\right]\right)\right).\label{eq:Fatoubound}
\end{align}
As seen above in Proposition~\ref{prop:momentWnprime}, the first term inside the summation on the right-hand side converges towards $-KR(X_u)e^{-X_u}$ as $n\to\infty$, almost surely.

We now need to control $\limsup_{n\to\infty}\mathbb{E}_{X_{u}}\left[a_n^{\alpha\bar\rho} W_{n-k_0}'\boldsymbol{1}_{a_n^{\alpha\bar\rho} W_{n-k_0}'\ge \varepsilon}\right]$.

\bigskip

Let $\mathcal{G}=\sigma\left(\xi_{k},X_{\xi_{k}i},k\in\mathbb{N},i\in\mathbb{N}^*\bigcup\{\varnothing\}\right)$ be the $\sigma$-algebra containing information about the spine and its children. Applying first the many-to-one formula, then Markov's inequality and finally Theorem~\ref{th:ScalingMin} we have:
\begin{align*}
\mathbb{E}_{x}\left[a_n^{\alpha\bar\rho} W_n'\boldsymbol 1_{a_n^{\alpha\bar\rho}W_n'\ge\varepsilon}\right]&=a_n^{\alpha\bar\rho}e^{-x}\mathbb{E}_x^*\left[\boldsymbol 1_{\min_{k\le n}X_{\xi_k}\ge0}\boldsymbol 1_{a_n^{\alpha\bar\rho}W_n'\ge\varepsilon}\right]\\
&=a_n^{\alpha\bar\rho}e^{-x}\mathbb{E}_x^*\left[\boldsymbol 1_{\min_{k\le n}X_{\xi_k}\ge 0}\mathbb{E}_x^*\left[\boldsymbol 1_{a_n^{\alpha\bar\rho}W_n'\ge\varepsilon}|\mathcal{G}\right]\right]\\
&\le a_n^{\alpha\bar\rho}e^{-x}\mathbb{P}_x^*\left(\min_{k\le n}X_{\xi_k}\ge0\right)\mathbb{E}_x^*\left[\mathbb{E}^*_x\left[\frac{a_n^{\alpha\bar\rho} W_n'}{\varepsilon}\wedge 1|\mathcal{G}\right]|\min_{k\le n}X_{\xi_k}\ge 0\right]\\
&\le c_2KR(x)e^{-x}\mathbb{E}_{x}^{*}\left[\mathbb{E}_{x}^*\left[\frac{a_n^{\alpha\bar\rho} W_n'}{\varepsilon}\wedge 1|\mathcal{G}\right]|\min_{k\le n}X_{\xi_k}\ge 0\right].
\end{align*}
Decomposing the particles according to their ancestor on the spine (see Lemma 4.1 in \cite{Boutaud2019}), we obtain
\begin{equation}
\mathbb{E}_{x}\left[a_n^{\alpha\bar\rho} W_n'\boldsymbol 1_{a_n^{\alpha\bar\rho}W_n'\ge\varepsilon}\right]\le c_2 KR(x)e^{-x}\left(T_{1}(x,\varepsilon,n)+T_2(x,\varepsilon,n)\right),
\end{equation}
where
\begin{align}
T_1(x,\varepsilon,n)&=\mathbb{E}_{x}^{*}\left[\frac{a_n^{\alpha\bar\rho} e^{-X_{\xi_n}}}{\varepsilon}\wedge 1\Big{|}\min_{k\le n}X_{\xi_k}\ge 0\right]\label{eq:defT1}\\
\text{and}\quad T_2(x,\varepsilon,n)&=\mathbb{E}_{x}^{*}\left[\left(\frac{a_n^{\alpha\bar\rho}}{\varepsilon}\sum_{k=0}^{n-1}\sum_{\substack{i\in\mathbb{N}\\\xi_ki\ne\xi_{k+1}}}\mathbb{E}_{X_{\xi_ki}}\left[W_{n-k-1}'\right]\right)\wedge 1\Big{|}\min_{k\le n}X_{\xi_k}\ge 0\right]\label{eq:defT2}.
\end{align}
This two terms are controlled by the following lemmas, whose proofs are postponed to Section~\ref{sec:proofT1T2}.
\begin{lemma}\label{lem:T1}
For any fixed $\eps>0$ and $x\ge 0$,
\begin{equation*}
T_1(x,\eps,n)\underset{n\to\infty}{\longrightarrow}0.
\end{equation*}
\end{lemma}

\begin{lemma}\label{lem:T2}
For every $\eps>0$, there exists a positive function $\tilde{h}$, such that $\tilde{h}(x)\to 0$ as $x\to\infty$ and such that the following holds: for every $x\ge 0$, we have
\begin{equation*}
\limsup_{n\to\infty}T_2(x,\eps,n)\le\tilde{h}(x).
\end{equation*}
\end{lemma}

Applying Lemma~\ref{lem:T1} and Lemma~\ref{lem:T2}, we obtain that
\begin{equation}\label{eq:boundedT1T2}
\limsup_{n\to\infty}\mathbb{E}_{x}\left[a_n^{\alpha\bar\rho} W_n'\boldsymbol 1_{a_n^{\alpha\bar\rho}W_n'\ge\varepsilon}\right]\le c_2 KR(x)e^{-x}\tilde{h}(x).
\end{equation}

As seen before, by Proposition~\ref{prop:momentWnprime}, the first term inside the summation of \eqref{eq:Fatoubound} converges towards $-KR(X_u)e^{-X_u}$ as $n\to\infty$, almost surely. Altogether we obtain from this, \eqref{eq:Fatoubound} and \eqref{eq:boundedT1T2}, that almost surely for every $k_0\in\mathbb{N}$,
\begin{equation}\label{eq:upper}
\limsup_{n\to\infty}\mathbb{E}\left[\exp\left(-s a_n^{\alpha\bar\rho}W_{n,k_0}''\right)\big{|}\mathcal{F}_{k_0}\right]\le\exp\left(s'\sum_{|u|=k_0}K(c_2\tilde{h}(X_u)-1)R(X_u)e^{-X_u}\right).
\end{equation}
Since $\min_{|u|=k_0}X_u\to\infty$ almost surely as $k_0\to\infty$ and $\tilde{h}(x)\to0$ as $x\to\infty$, we get almost surely,
\begin{equation}
\label{eq:upper2}
\lim_{k_0\to\infty}\sum_{|u|=k_0}K(c_2\tilde{h}(X_u)-1)R(X_u)e^{-X_u}=\lim_{k_0\to\infty}-\sum_{|u|=k_0}KR(X_u)e^{-X_u}=-\kappa Z_\infty,
\end{equation}
by \eqref{eq:lower2}.
Together with \eqref{eq:upper}, this shows that
\begin{equation}\label{eq:upper3}
\limsup_{k_0\to\infty}\limsup_{n\to\infty}\mathbb{E}\left[\exp(-s a_n^{\alpha\bar\rho}W_{n,k_0}'')\Big{|}\FF_{k_0}\right]\le\exp\left(-s'\kappa Z_\infty\right),\quad\text{a.s.}
\end{equation}
Letting $s'\to s$ in \eqref{eq:upper3}, and using \eqref{eq:lower3}, we finally get for any $s>0$,
\begin{align*}
\lim_{k_0\to\infty}\liminf_{n\to\infty}\mathbb{E}\left[\exp(-s a_n^{\alpha\bar\rho}W_{n,k_0}'')\Big{|}\FF_{k_0}\right]&=\lim_{k_0\to\infty}\limsup_{n\to\infty}\mathbb{E}\left[\exp(-s a_n^{\alpha\bar\rho}W_{n,k_0}'')\Big{|}\FF_{k_0}\right]\\
&=\exp\left(-s\kappa Z_\infty\right),\quad\text{a.s.}
\end{align*}

Using Cantor diagonal extraction, there exists a sequence $(k_0(n))_{n\ge 0}$ (that goes to infinity as $n\to\infty$) such that for any $s\in\mathbb{Q}\cap(0,\infty)$, $\mathbb{E}\left[\exp(-s a_n^{\alpha\bar\rho}W_{n,k_0(n)}'')\Big{|}\FF_{k_0(n)}\right]$ converges to $\exp\left(-s\kappa Z_\infty\right)$ almost surely as $n\to\infty$. Then applying Lemma B.1 in \cite{Boutaud2019} with $Y_n=a_n^{\alpha\bar\rho}W_{n,k_0(n)}''$ and $\GG_n=\FF_{k_0(n)}$ we obtain
\begin{equation*}
a_n^{\alpha\bar\rho}W_{n,k_0(n)}''\underset{n\to\infty}{\longrightarrow}\kappa Z_\infty\quad\text{in probability}.
\end{equation*}

We conclude by using Equation~\eqref{eq:compareWnsecond} to see that $a_n^{\alpha\bar\rho}W_n$ converges in probability to $\kappa Z_\infty$ as $n\to\infty$.
\end{proof}

\section{Proofs of Lemma~\ref{lem:T1} and Lemma~\ref{lem:T2}}\label{sec:proofT1T2}

We know that $(X_{\xi_{\lfloor nt\rfloor}}/a_n)_t$ converges in distribution towards an $\alpha$-stable process under $\mathbb{P}^*_x$ as $n\to\infty$ and we will need to discuss what happens to the convergence of this process if we condition the walk $(X_{\xi_n})_n$ to stay non-negative up to some time $n$ or forever.
The proof of Lemma~\ref{lem:T1} relies on the behaviour of the rescaled process conditioned to stay non-negative up to time $n$.
\begin{proof}[Proof of Lemma~\ref{lem:T1}]
~\\
From equation (3.3) in Caravenna and Chaumont \cite{Caravenna2008}, we know that the rescaled process $(X_{\xi_{\lfloor nt\rfloor}}/a_n)_t$ under $\mathbb{P}^*(\bullet|\min_{k\le n}X_{\xi_k}\ge0)$, i.e.~ conditioned to stay non-negative up to time $n$, converges in distribution to the law $\mathbb{P}^{(m)}$ of the meander at time $1$.
Remembering that $0\le\ag\bar\rho\le 1$, we obtain that $a_n^{\alpha\bar\rho}e^{-X_{\xi_n}}$ converges to $0$ under the same conditioning. Moreover, the random variables $\left(\frac{a_n^{\alpha\bar\rho}}{\varepsilon}e^{-X_{\xi_n}}\right)\wedge 1$ are trivially bounded by $1$. Hence, using the dominated convergence theorem and recalling the definition of $T_1(x,\varepsilon,n)$ in \eqref{eq:defT1},
\begin{equation*}
T_{1}(x,\varepsilon,n)\underset{n\to\infty}{\longrightarrow}0.
\end{equation*}
\end{proof}

\begin{proof}[Proof of Lemma~\ref{lem:T2}]
~\\
Let $\varepsilon>0$ and suppose, without loss of generality, that $\varepsilon<1$.
By Proposition~\ref{prop:momentWnprime}, we have
\begin{align*}
&\frac{a_n^{\alpha\bar\rho}}{\varepsilon}\sum_{k=0}^{n-1}\sum_{\substack{i\in\mathbb{N}\\\xi_{k}i\ne\xi_{k+1}}}\mathbb{E}_{X_{\xi_ki}}[W_{n-k-1}']\\
&\le\frac{c_{2}K}{\varepsilon}\sum_{k=0}^{n-1}\frac{a_n^{\alpha\bar\rho}}{a_{n-k-1}^{\alpha\bar\rho}}\sum_{\substack{i\in\mathbb{N}\\\xi_{k}i\ne\xi_{k+1}}}R(X_{\xi_ki})e^{-X_{\xi_ki}}\\
&\le\frac{c_{2}K}{\varepsilon}\sum_{k=0}^{n-1}\frac{a_n^{\alpha\bar\rho}}{a_{n-k-1}^{\alpha\bar\rho}}\sum_{\substack{i\in\mathbb{N}\\\xi_{k}i\ne\xi_{k+1}}}R(X_{\xi_k}+X_{\xi_ki}-X_{\xi_k})e^{-X_{\xi_k}-(X_{\xi_ki}-X_{\xi_k})}.
\end{align*}
Define for every $k\in\mathbb N$,
\begin{equation*}
V_k=\sum_{\substack{i\in\mathbb{N}\\\xi_{k}i\ne\xi_{k+1}}}\left(1+(X_{\xi_ki}-X_{\xi_k})^+\right)^{\ag\bar\rho}e^{-(X_{\xi_ki}-X_{\xi_k})}.
\end{equation*}
Then by Lemma~\ref{lem:sousadditif}, we get
\begin{align*}
\frac{a_n^{\alpha\bar\rho}}{\varepsilon}\sum_{k=0}^{n-1}\sum_{\substack{i\in\mathbb{N}\\\xi_{k}i\ne\xi_{k+1}}}\mathbb{E}_{X_{\xi_ki}}[W_{n-k-1}']&\le\frac{c_{2}c_{1}'K}{\varepsilon}\sum_{k=0}^{n-1}\frac{a_n^{\alpha\bar\rho}}{a_{n-k-1}^{\alpha\bar\rho}}(1+X_{\xi_k}^+)^{\ag\bar\rho}e^{-X_{\xi_k}}V_k\\
&\le\frac{c_{7}}{\varepsilon}\sum_{k=0}^{n-1}\frac{a_n^{\alpha\bar\rho}}{a_{n-k-1}^{\alpha\bar\rho}}e^{-X_{\xi_k}/2}V_k,
\end{align*}
where $c_{7}>0$ is such that for all $x\in\mathbb{R}$, $c_{2}c_{1}'K(1+x^+)^{\ag\bar\rho}e^{-x}\le c_{7}e^{-x/2}$.

Recalling the definition of $T_2$ in \eqref{eq:defT2}, putting $Y_n=\left(\sum_{k=0}^{n-1}\left(\frac{a_n^{\alpha\bar\rho}}{a_{n-k-1}^{\alpha\bar\rho}}\frac{e^{-X_{\xi_k}/2}V_k}{\varepsilon}\wedge 1\right)\right)\wedge 1$ and using the previous inequalities yields for some constant $c_{8}>0$
\begin{equation}\label{eq:boundT2}
T_{2}(x,\varepsilon,n)\le c_{8}\mathbb{E}_{x}^{*}\left[Y_n\Big{|}\min_{k\le n}X_{\xi_k}\ge 0\right].
\end{equation}

In order to bound this expectation, we first bound $Y_n$.

For every $0\le k\le\lfloor n/2\rfloor-1$, there exists $s_k\in[1/2,1]$ such that $n-k-1=s_k n$. By assumption on the sequence $(a_n)_n$, there exists a function $\nu$, regularly varying at $\infty$ with index $1/\alpha$ such that $a_n=\nu(n)$. By Karamata's uniform convergence theorem, we have
\begin{equation*}
\forall s>0, \frac{\nu(x)^{\alpha\bar{\rho}}}{\nu(sx)^{\alpha\bar{\rho}}}\underset{x\to\infty}{\longrightarrow}s^{-\bar{\rho}},\quad\text{uniformly for $s$ in compact sets}.
\end{equation*}
This implies that for $n$ large enough, 
$$\frac{a_n^{\alpha\bar\rho}}{a_{n-k-1}^{\alpha\bar\rho}} \le 2^{\bar{\rho}}+1 \le 3,$$
and therefore,
\begin{equation*}
\left(\sum_{k=0}^{\lfloor n/2\rfloor}\left(\frac{a_n^{\alpha\bar\rho}}{a_{n-k-1}^{\alpha\bar\rho}}\frac{e^{-X_{\xi_k}/2}V_k}{\varepsilon}\wedge 1\right)\right)\wedge 1\le 3Y_n',
\end{equation*}
where
\begin{equation*}
Y_n'=\left(\sum_{k=0}^{\lfloor n/2\rfloor}\left(\frac{e^{-X_{\xi_k}/4}V_k}{\varepsilon}\wedge 1\right)\right)\wedge 1
\end{equation*}
(the factor $1/4$ is chosen for later convenience).
Thus we obtain for $n$ large enough, using furthermore that $a_n \ge 1$ for large $n$,
\begin{equation}\label{eq:boundYn}
Y_n\le 3Y_n'+Y_n'',
\end{equation}
where
\begin{equation*}
Y_n''=\left(\sum_{k=\lfloor n/2\rfloor}^{n-1}\left(\frac{a_n^{\alpha\bar\rho}e^{-X_{\xi_k}/2}V_k}{\varepsilon}\wedge 1\right)\right)\wedge 1.
\end{equation*}

By monotone convergence, we have for every $x\ge0$ that $Y_n'$ converges $\mathbb{P}_x^+$-almost surely as $n\to\infty$ to $Y_{\infty}'$ defined as
\begin{equation}
Y_\infty'=\left(\sum_{k=0}^{\infty}\left(\frac{e^{-X_{\xi_k}/4}V_k}{\varepsilon}\wedge 1\right)\right)\wedge 1.
\end{equation}
We now claim the following:
\begin{enumerate}
 \item[a)] $\mathbb{E}_x^+[Y_\infty']$ is finite for every $x\ge0$ and $\mathbb{E}_x^+[Y_\infty'] \to 0$ as $x\to\infty$.
 \item[b)] $Y_n'' \to 0$ as $n\to\infty$, in $\mathbb{P}_x^+$-probability.
\end{enumerate}
These two claims imply the statement of the lemma. Indeed, first applying Lemma~5.1 from Boutaud and Maillard \cite{Boutaud2019} to the r.v.'s $(Y_n')_{n\ge 0}$ and $Y_{\infty}'$, we have for every $x\ge0$,
\begin{equation*}
\lim_{n\to\infty}\mathbb{E}_{x}^{*}\left[Y_n'|\min_{k\le n}X_{\xi_k}\ge 0\right]=\mathbb{E}_{x}^+\left[Y_{\infty}'\right].
\end{equation*}
Second, applying again Lemma~5.1 from \cite{Boutaud2019} to the r.v.'s $(Y_n'')_{n\ge 0}$ and using Claim b), we get
\begin{equation*}
\lim_{n\to\infty}\mathbb{E}_{x}^{*}\left[Y_n''|\min_{k\le n}X_{\xi_k}\ge 0\right]=0.
\end{equation*}
Now plugging these two equalities into \eqref{eq:boundT2} and \eqref{eq:boundYn} gives us with some $c_{9}>0$,
\begin{equation*}
\limsup_{n\to\infty}T_2(x,\varepsilon,n)\le c_{9}\mathbb{E}_{x}^{+}\left[Y_\infty'\right].
\end{equation*}
Together with Claim a), this yields the lemma.

We now prove Claims a) and b), starting with Claim a).

Using the bound \eqref{eq:biasedbound}, for every $k\ge 0$, we get
\begin{multline}
 \qquad\mathbb{E}_{x}^+\left[\frac{e^{-X_{\xi_k}/4}V_k}{\eps}\wedge 1\big{|}\FF_{k}\right]\le \frac{c_{3}}{\varepsilon} f(X_{\xi_k}),\quad\text{where}\\
 f(y) = \mathbb{E}\left[\left(W_1+\frac{Z_1}{R(y)}\right)\left(e^{-y/4}(W_1+Z_1)\wedge 1\right)\right].\qquad\label{eq:deff}
\end{multline}
We decompose $f$:
\begin{equation*}
 f(y) = f_{1}(y)+\frac{f_{2}(y)}{R(y)},
\end{equation*}
with $f_1$ and $f_2$ defined in \eqref{eq:deff1} and \eqref{eq:deff2}.
Using \eqref{eq:f1f2integralbound} and the bound \eqref{eq:boundR} on $R$, we obtain
\begin{equation*}
\int_{0}^{\infty}f(y)y^{\ag-1}dy<\infty.
\end{equation*}




Now we compute:
\begin{equation*}
\mathbb{E}_{x}^{+}\left[Y_\infty'\right]\le\sum_{k=0}^{\infty}\mathbb{E}_{x}^+\left[\frac{e^{-X_{\xi_k}/4}V_k}{\varepsilon}\wedge 1\right]\le \frac{c_{3}}{\varepsilon}\mathbb{E}_{x}^{+}\left[\sum_{k=0}^{\infty}f\left(X_{\xi_k}\right)\right].
\end{equation*}

Using Lemma~\ref{lem:GreenEstimate}, we obtain that
\begin{equation*}
\mathbb{E}_{x}^{+}\left[\sum_{k=0}^{\infty}f\left(X_{\xi_k}\right)\right]\to 0,
\end{equation*}
and the expectation on the left-hand side is finite for every $x\ge 0$. This implies Claim a).

To prove Claim b) we use an invariance principle by Caravenna and Chaumont \cite{Caravenna2008}: for every $x\ge0$, as $n\to\infty$, the rescaled process $\left(\frac{X_{\xi_{\lfloor nt\rfloor}}}{a_n}\right)_{t\ge 0}$ converges in distribution under $\mathbb{P}^+_x$ to a non-degenerate limit, independent of $x$, and which is strictly positive at all times $t>0$ (this limit can be interpreted as the process $\mathcal X$ conditioned to stay positive for all times). As a consequence, for every $\eta\in(0,1)$, there exists $\dg>0$ such that for large $n$, with probability at least $1-\eta$, we have $X_{\xi_k}>\dg a_n$ for every $k\in\{\lfloor n/2\rfloor,\ldots,n-1\}$. So there is some positive constant $c_{10}$ such that, with probability at least $1-\eta$,
\begin{align*}
Y_n''&\le\left(\sum_{k=\lfloor n/2\rfloor}^{n-1}\left(\frac{a_n^{\alpha\bar\rho} e^{-X_{\xi_k}/2}V_k}{\varepsilon}\wedge 1\right)\right)\wedge 1\\
&\le c_{10}\sum_{k=\lfloor n/2\rfloor}^{n-1}\left(\frac{e^{-X_{\xi_k}/4}V_k}{\varepsilon}\wedge 1\right),
\end{align*}
which converges to $0$ in $\mathbb{P}^+_x$ probability as $n\to\infty$ since for every $x\ge 0$, as shown above,
\begin{equation*}
\mathbb{E}_{x}^{+}\left[\sum_{k=0}^{\infty}\left(\frac{e^{-X_{\xi_k}/4}V_k}{\varepsilon}\wedge 1\right)\right]<\infty.
\end{equation*}
This ends the proof of Claim b) and thus of the lemma.
\end{proof}

\section{Proof of Equation (\ref{eq:kappaexplicit})}\label{sec:computingkappa}
This section is devoted to the proof of Equation (\ref{eq:kappaexplicit}). Recall the $\alpha$-stable process $\mathcal X$ defined in the introduction. We assume $\lambda=1$, where $\lambda$ is the parameter in \eqref{eq:caracLevy}. Following Caravenna and Chaumont \cite{Caravenna2008}, one can define for every $x>0$ a probability measure $\mathbb P^+_x$ such that for all $t\ge0$,
\[
\frac{d\mathbb P^+_x}{d\mathbb P_x}\Big|_{\FF_t} = \frac{\mathcal X_1^{\alpha\bar{\rho}}}{x^{\alpha\bar{\rho}}},
\]
where $(\FF_t)_{t\ge0}$ denotes here the canonical filtration of the process $\mathcal X$. 
Furthermore, the weak limit $\mathbb P^+ = \mathbb P^+_0 = \lim_{x\downarrow0}\mathbb P^+_x$ exists. This probability law is related to the law of the meander at time $1$ by:
\begin{equation}
\frac{d\mathbb{P}^+}{d\mathbb{P}^{(m)}}\Big|_{\FF_1} =\kappa \mathcal{X}_1^{\ag\bar\rho}.
\end{equation}
From this we deduce that
\begin{equation}\label{eq:kappaplus}
\kappa=\mathbb{E}^+\left[\frac{1}{\mathcal{X}_{1}^{\ag\bar\rho}}\right].
\end{equation}

Now we suppose that $\alpha\in(1,2]$ and $\ag\rho=1$, so that the Lévy process $\mathcal{X}$ has no positive jumps. For all $s\ge0$, put
\begin{equation}
\psi(s)=\log\mathbb{E}\left[e^{s\mathcal{X}_1}\right].
\end{equation}
Using Theorem 2.6.1 from Zolotarev \cite{Zolotarev}, we have the following explicit expression of $\psi$:
\begin{equation}
\psi(s)=s^{\ag}.
\end{equation}
In particular, $0$ is the only solution to $\psi(s)=0$. Then following Bertoin \cite[Chapter VII]{Bertoin1996}, define $W:[0,\infty)\to[0,\infty)$ the scale function with the following characterization: $W$ is the unique absolutely continuous increasing function with Laplace transform
\begin{equation}\label{eq:scalefunction}
\forall s>0, \int_{0}^{\infty}e^{-sx}W(x)dx=\frac{1}{\psi(s)}=\frac{1}{s^{\ag}}.
\end{equation}
Using Equation~\eqref{eq:scalefunction} and the fact that when $\ag\rho=1$ we have $\ag-1=\ag\bar\rho$, we obtain that
\begin{equation*}
W(x)=\frac{x^{\ag\bar\rho}}{\Gamma(\ag)}.
\end{equation*}
Using Corollary 16 from Bertoin \cite[Chapter VII]{Bertoin1996} together with equation~\eqref{eq:kappaplus}, we get
\begin{align}
\kappa&=\mathbb{E}\left[\frac{\mathcal{X}_1W\left(\mathcal{X}_1\right)}{\mathcal{X}_1^{\ag\bar\rho}}\boldsymbol 1_{\mathcal{X}_1>0}\right].\\
&=\frac{1}{\Gamma(\ag)}\mathbb{E}\left[\mathcal{X}_1\boldsymbol 1_{\mathcal{X}_1>0}\right].\label{eq:kappaZolo}
\end{align}
We proceed by computing $\mathbb{E}\left[\mathcal{X}_1\boldsymbol 1_{\mathcal{X}_1>0}\right]$. Using Theorem~2.6.2 from Zolotarev \cite{Zolotarev}, we have for all $s\ge0$:
\begin{align*}
\mathbb{E}\left[\mathcal{X}_1e^{-s\mathcal{X}_1}\boldsymbol 1_{\mathcal{X}_1>0}\right]&=\frac{1}{\pi}\int_{0}^{\infty}\ag s^{\ag-1}u^\ag e^{-(su)^\ag}\frac{\sin(\pi\rho)}{u^2+2u\cos(\pi\rho)+1}du\\
&=\frac{\ag\sin(\pi\rho)}{\pi}\int_{0}^{\infty}v^\ag e^{-v^\ag}\frac{1}{v^2+2sv\cos(\pi\rho)+s^2}dv,
\end{align*}
where we used the change of variables $v=su$. We then make $s\to 0$ using dominated convergence and get, changing again variables,
\begin{align*}
\mathbb{E}\left[\mathcal{X}_1\boldsymbol 1_{\mathcal{X}_1>0}\right]&=\frac{\ag\sin(\pi\rho)}{\pi}\int_{0}^{\infty}v^{\ag-2}e^{-v^\ag}dv\\
&=\frac{\sin(\pi\rho)}{\pi}\int_{0}^{\infty}t^{-1/\ag}e^{-t}dt\\
&=\frac{\sin(\pi\rho)}{\pi}\Gamma\left(1-\frac{1}{\ag}\right)\\
&=\frac{\sin\left(\frac{\pi}{\alpha}\right)}{\pi}\Gamma\left(1-\frac{1}{\ag}\right) && \text{(since $\alpha\rho=1$)}\\
&=\frac{1}{\Gamma(1/\ag)}&&\text{(by Euler's reflection formula)}.
\end{align*}
Finally, this yields together with equation \eqref{eq:kappaZolo}:
\begin{equation*}
\kappa=\frac{1}{\Gamma(\alpha)\Gamma(1/\ag)},
\end{equation*}
which is Equation~\eqref{eq:kappaexplicit}.

\bibliographystyle{alpha}
\bibliography{these_pierre}

\begin{thebibliography}{BHM18}

\bibitem[A{\"{i}}d13]{Aidekon2013}
Elie A{\"{i}}d{\'{e}}kon.
\newblock {Convergence in law of the minimum of a branching random walk}.
\newblock {\em Annals of Probability}, 41(3A):1362--1426, 2013.

\bibitem[AS14]{Aidekon2014}
Elie Aid{\'{e}}kon and Zhan Shi.
\newblock {The Seneta-Heyde scaling for the branching random walk}.
\newblock {\em Annals of Probability}, 42(3):959--993, 2014.

\bibitem[BD94]{Bertoin1994}
Jean Bertoin and Ronald~A Doney.
\newblock {On conditioning a random walk to stay nonnegative.}
\newblock {\em Ann. Probab.}, 22(4):2152--2167, 1994.

\bibitem[Ber96]{Bertoin1996}
Jean Bertoin.
\newblock {\em {L{\'{e}}vy processes}}, volume 121 of {\em Cambridge Tracts in
  Mathematics}.
\newblock Cambridge University Press, Cambridge, 1996.

\bibitem[Ber19]{Berger2018}
Quentin Berger.
\newblock {Notes on random walks in the Cauchy domain of attraction}.
\newblock {\em Probability Theory and Related Fields}, 175(1-2), 2019.

\bibitem[BHM18]{Barral2014}
Julien Barral, Yueyun Hu, and Thomas Madaule.
\newblock {The minimum of a branching random walk outside the boundary case}.
\newblock {\em Bernoulli}, 24(2):801--841, 2018.

\bibitem[Big77a]{Biggins1977a}
J.~D. Biggins.
\newblock {Chernoff's theorem in the branching random walk}.
\newblock {\em Journal of Applied Probability}, 14(3):630--636, 1977.

\bibitem[Big77b]{Biggins1977}
J.~D. Biggins.
\newblock {Martingale convergence in the branching random walk}.
\newblock {\em Journal of Applied Probability}, 14(01):25--37, 1977.

\bibitem[Big98]{Biggins1998}
John~D. Biggins.
\newblock {Lindley-type equations in the branching random walk.}
\newblock {\em Stochastic Processes Appl.}, 75(1):105--133, 1998.

\bibitem[Bin73a]{Bingham1973a}
N.~H. Bingham.
\newblock {Limit Theorems in Fluctuation Theory.}
\newblock {\em Advances in Applied Probability}, 5(3):554--569, 1973.

\bibitem[Bin73b]{Bingham1973}
N.~H. Bingham.
\newblock {Maxima of sums of random variables and suprema of stable processes}.
\newblock {\em Zeitschrift f{\"{u}}r Wahrscheinlichkeitstheorie und Verwandte
  Gebiete}, 26(4):273--296, 1973.

\bibitem[BK04]{Biggins2004}
J.~D. Biggins and A.~E. Kyprianou.
\newblock {Measure change in multitype branching}.
\newblock {\em Advances in Applied Probability}, 36(2):544--581, 2004.

\bibitem[BK05]{BK2005}
John~D. Biggins and Andreas~E. Kyprianou.
\newblock {Fixed Points of the Smoothing Transform: the Boundary Case}.
\newblock {\em Electronic Journal of Probability}, 10:609--631, 2005.

\bibitem[BM19]{Boutaud2019}
Pierre Boutaud and Pascal Maillard.
\newblock {A revisited proof of the Seneta-Heyde norming for branching random
  walks under optimal assumptions}.
\newblock {\em Electronic Journal of Probability}, 24, 2019.

\bibitem[Bov17]{BovierBook}
Anton Bovier.
\newblock {\em {Gaussian processes on trees}}, volume 163 of {\em Cambridge
  Studies in Advanced Mathematics}.
\newblock Cambridge University Press, Cambridge, 2017.

\bibitem[CC08]{Caravenna2008}
Francesco Caravenna and Lo{\"{i}}c Chaumont.
\newblock {Invariance principles for random walks conditioned to stay
  positive}.
\newblock {\em Annales de l'institut Henri Poincare (B) Probability and
  Statistics}, 44(1):170--190, 2008.

\bibitem[Che15]{Chen2015}
Xinxin Chen.
\newblock {A necessary and sufficient condition for the nontrivial limit of the
  derivative martingale in a branching random walk}.
\newblock {\em Advances in Applied Probability}, 47(3):741--760, 2015.

\bibitem[Eme72]{Emery1972}
D.~J. Emery.
\newblock {Limiting behaviour of the distributions of the maxima of partial
  sums of certain random walks}.
\newblock {\em Journal of Applied Probability}, 9(03):572--579, sep 1972.

\bibitem[Fel71]{Feller1971}
William Feller.
\newblock {\em {An introduction to probability theory and its applications. Vol
  II.}}
\newblock John Wiley \& Sons Inc., New York, second edition, 1971.

\bibitem[Hey69]{Heyde1969}
C.C. Heyde.
\newblock {On the maximum of sums of random variables and the supremum
  functional for stable processes}.
\newblock {\em Journal of Applied Probability}, 6(02):419--429, aug 1969.

\bibitem[HLZ18]{He2016}
Hui He, Jingning Liu, and Mei Zhang.
\newblock {On Seneta-Heyde Scaling for a stable branching random walk}.
\newblock {\em Advances in Applied Probability}, 50:565--599, 2018.

\bibitem[Koz76]{Kozlov1976}
M.~V. Kozlov.
\newblock {On the asymptotic behavior of the probability of non-extinction for
  critical branching processes in a random environment}.
\newblock {\em Theory of probability and its applications}, 21(4):2037, 1976.

\bibitem[Kyp04]{Kyprianou2004}
Andreas~E. Kyprianou.
\newblock {Travelling wave solutions to the {K-P-P} equation: Alternatives to
  {S}imon {H}arris' probabilistic analysis}.
\newblock {\em Annales de l'Institut Henri Poincar{\'{e}} (B)
  Probabilit{\'{e}}s et Statistiques}, 40(1):53--72, 2004.

\bibitem[Lyo97]{Lyons1998}
Russell Lyons.
\newblock {A Simple Path to Biggins' Martingale Convergence for Branching
  Random Walk}.
\newblock In K.B. Athreya and Peter Jagers, editors, {\em Classical and Modern
  Branching Processes}, The IMA Volumes in Mathematics and its Applications,
  vol 84, pages 217--221. Springer, New York, NY, 1997.

\bibitem[RV14]{RhodesVargasReview}
R{\'{e}}mi Rhodes and Vincent Vargas.
\newblock {Gaussian multiplicative chaos and applications: A review}.
\newblock {\em Probability Surveys}, 11:315--392, 2014.

\bibitem[Shi15]{ShiLectureNotes}
Zhan Shi.
\newblock {\em {Branching random walks}}, volume 2151 of {\em Lecture Notes in
  Mathematics. {\'{E}}cole d'{\'{E}}t{\'{e}} de Probabilit{\'{e}}s de
  Saint-Flour XLII -- 2012}.
\newblock Springer, Cham, 2015.

\bibitem[Tan89]{Tan1989}
Hiroshi Tanaka.
\newblock {Time reversal of random walks in one dimension.}
\newblock {\em Tokyo J. Math.}, 12(1):159--174, 1989.

\bibitem[VD17]{Vatutin2017}
Vladimir Vatutin and Elena Dyakonova.
\newblock {Path to survival for the critical branching processes in a random
  environment}.
\newblock {\em Journal of Applied Probability}, 54(2):588--602, 2017.

\bibitem[Zei16]{ZeitouniLNBRW}
Ofer Zeitouni.
\newblock {Branching random walks and Gaussian fields}.
\newblock In {\em Probability and statistical physics in St. Petersburg},
  volume~91 of {\em Proc. Sympos. Pure Math.}, pages 437--471. Amer. Math.
  Soc., Providence, RI, 2016.

\bibitem[Zol86]{Zolotarev}
V.~M. Zolotarev.
\newblock {\em {One-dimensional stable distributions}}.
\newblock Translations of Mathematical Monographs, Volume 65. American
  Mathematical Society, Providence, Rhode Island, 1986.

\end{thebibliography}

\end{document}